\documentclass[11pt]{amsart}
\usepackage{amsmath}
\usepackage{amssymb}
\usepackage{graphicx}
\usepackage{tikz}
 \usepackage{psfrag}

\newtheorem{theorem}{Theorem}[section]

\newtheorem{lemma}[theorem]{Lemma}
\newtheorem{proposition}[theorem]{Proposition}
\newtheorem{problem}[theorem]{Problem}

\theoremstyle{definition}
\newtheorem{definition}{Definition}[section]

\newtheorem{remark}{Remark}[section]

\numberwithin{equation}{section}

\newcommand{\B}{\mathbb B}
\newcommand{\R}{\mathbb R}

\newcommand{\K}{\mathcal K}
\newcommand{\U}{\mathcal U}

\renewcommand{\u}{\mathbf u}

\newcommand{\z}{\mathbf z}
\newcommand{\x}{\mathbf x}
\newcommand{\y}{\mathbf y}

\newcommand{\0}{\mathbf 0}

\newcommand{\tr}{\operatorname{tr}}
\newcommand{\rank}{\operatorname{rank}}

\newcommand{\dv}{\operatorname{div}}
\renewcommand{\mod}{\operatorname{mod}}

\newcommand{\adj}{\operatorname{adj}}

\title{The $\tau_N$-configurations and polyconvex gradient flows}

 \author{Baisheng Yan}

\address{Department of Mathematics\\ Michigan State University\\ East Lansing, MI 48824, USA}
   \email{\tt yanb@msu.edu}



\date{\today}

\begin{document}

 \begin{abstract} We study a generalization  of  $T_N$-configurations, called the $\tau_N$-configurations,   for constructing certain irregular solutions of some nonlinear diffusion systems by the  method of convex integration. We  construct some polyconvex functions  that support  a parametrized family of  $\tau_N$-configurations  satisfying a general openness condition; this will guarantee   the existence of   nowhere-$C^1$ Lipschitz weak solutions to the initial boundary value problems of the polyconvex gradient flows.  We elaborate on such  constructions and the subsequent verification of the openness condition when the dimension is at least 4 to avoid some complicated calculations that cannot be done by hand but would otherwise  be needed for dimensions 2 and 3.  
     \end{abstract}

\maketitle


\section{Introduction}
Let $m,n\ge 1$ and $\R^{m\times n}$ be the usual space of real $m\times n$  matrices with inner-product $\langle A,B\rangle=\tr(A^TB).$ Given a bounded  domain $\Omega\subset\R^n$, a number $T>0$ and a continuous function $\sigma\colon \R^{m\times n}\to \R^{m\times n},$ we study the time-dependent diffusion  system: 
\begin{equation}\label{DE}
 \partial_t u=\dv \sigma(Du )   \quad  \mbox{in $\Omega_T=\Omega \times (0,T),$}
\end{equation}
for the unknown $u=(u^1,\dots,u^m)\colon\Omega_T\to \R^m$, where
$\partial_t u=(\partial_t u^1,\cdots, \partial_t u^m) $ and $Du=\big(\frac{\partial u^i}{\partial {x_j}}\big) $ are the time-derivative and space-gradient of $u,$ respectively.  
We say that  a function $
 u\in L_{loc}^1(0,T; W_{loc}^{1,1}(\Omega;\R^m))$ is a (very) {\em weak solution} of  (\ref{DE}) provided that $\sigma(Du)\in L_{loc}^1(\Omega_T;\R^{m\times n})$  and
\begin{equation}\label{weak-sol}
\int_{\Omega_T} \Big ( u \cdot \partial_t  \varphi  - \langle \sigma(D u), D\varphi \rangle \Big ) dxdt =0\quad \forall\, \varphi\in  C^\infty_c(\Omega_T;\R^m).
\end{equation}

The system  (\ref{DE}) is  called (uniformly) {\it parabolic}  if $\sigma$ satisfies  a uniform rank-one monotonicity (also known as Legendre-Hadamard ellipticity) condition:  
 \begin{equation}\label{r-mono}
 \langle \sigma(A+p\otimes a) -\sigma(A), \, p\otimes a\rangle \ge \nu |p|^2|a|^2 
  \end{equation}
for all $A\in \R^{m\times n},\, p\in\R^m $ and $a\in \R^n,$ where  $\nu>0$ is a constant.  However, it is known that condition (\ref{r-mono})  is not enough for any feasible theory of existence and regularity for system (\ref{DE}); further stronger structural conditions would be needed. For example, under a stronger  condition known as {\em quasimonotonicity}, the best regularity  result for weak solutions of    (\ref{DE}) is  the {\it partial $C^{1,\alpha}$-regularity} in space:  $Du\in C_{loc}^{\alpha,\alpha/2}(U;\R^{m\times n})$ for some $0<\alpha<1$ and  open set $U\subset\Omega_T$ with $|\Omega_T\setminus U|=0$ (see, e.g., \cite{BDM13}). 
 
We address the issues on counterexamples to uniqueness and regularity for system (\ref{DE}).  To construct such counterexamples,  a general approach  is to reformulate  system (\ref{DE})  as a first-order  partial differential relation: 
 \begin{equation}\label{pdr1}
u= \dv v,\;\;  (Du,\partial_t v)\in \K \;\; \mbox{a.e. on $\Omega_T$}
 \end{equation}
 for functions  $(u,v)\colon\Omega_T\to\R^m\times   \R^{m\times n},$    where 
\[
\K  =\{(A,\sigma(A)):A\in \R^{m\times n}\} 
\]
is the graph of $\sigma$ in $\R^{m\times n}\times \R^{m\times n},$ and then apply the method of convex integration. Such an approach has been successful  if $\sigma$ satisfies some non-monotonicity conditions in the scalar cases ($m=1$) as studied in \cite{KY15, KY17, KY18,  Zh06a, Zh06b} or if $\sigma$ satisfies even some conditions much stronger than the rank-one monotonicity in the system cases  ($m,n\ge 2$) as studied in \cite{Ya20, Ya22}.  Earlier work on the corresponding elliptic systems and certain  time-dependent solutions of parabolic systems has been pioneered  in \cite{MRS05, MSv03, Sz04} using the convex integration method. We remark that the convex integration method has  recently found remarkable success  in  studying many other important PDE problems,  such as  the incompressible Euler  and Navier-Stokes equations  \cite{BV19, DLSz09,  DLSz17,  Is18}, the porous medium equations \cite{CFG11},   the  Monge-Amp\`ere equations \cite{LP17},    the active scalar equations \cite{Sh11},  and most recently, the wild solutions to elliptic equations including the $p$-Laplacian equation  \cite{CT22, Jo23}, to just list a few. 
 
In this paper, we focus on the systems given by $\sigma(A)=DF(A)$ with {\it strongly polyconvex} functions $F\colon \R^{m\times n}\to\R$  in the sense that  
\begin{equation} 
F(A)=\frac{\nu}{2}|A|^2 +G(\Gamma (A)),
\end{equation}
where  $\nu>0$, $\Gamma (A) \in \R^r$ is an ordered $r$-tuple of certain subdeterminants of $A,$ and $G\colon\R^{r} \to\R$ is a smooth convex function. For such a function $F,$ condition (\ref{r-mono}) is automatically satisfied for $\sigma=DF,$ and  system (\ref{DE}) becomes the $L^2$-gradient flow of energy functional $I(u)=\int_\Omega F(Du(x))\,dx.$

Some basic structures of the set $\K$ that are useful for the convex integration of (\ref{pdr1}) can be elucidated  by the rank-one matrices in  the full space-time  gradient space of $(u,v).$   However,  condition (\ref{r-mono}) rules out nontrivial rank-one connections in a  full gradient set that contains $\K$ as its diagonal projection, and thus other rank-one convex hull structures are needed.  Indeed, as in \cite{Ya20}, we consider  the  {\it $T_N$-configurations} in the full  gradient space and then project them onto the diagonal to obtain the so-called {\em $\tau_N$-configurations} in the space of $\K.$ 
We remark that the $T_4$-configuration was discovered independently by several authors in different contexts and  was used in \cite{Ta93} to study  separately convex functions. The $T_4$-configuration  and its generalization of $T_N$-configurations (see \cite{KMS03}) have been greatly explored and applied in  constructing wild solutions for elliptic systems in \cite{MSv03, Sz04} and for parabolic systems in  \cite{MRS05, Ya20}.  

 Based on the work \cite{Ya20}, we introduce a structural condition, called {\em Condition $(OC)_N$} (see Definition \ref{OC-N}), which requires that certain parametrized family of $\tau_N$-configurations be supported on the set $\K$  and satisfy some openness conditions.  It turns out that this condition is sufficient for constructing nowhere-$C^1$ Lipschitz weak solutions of system (\ref{DE}); the proof of such a sufficiency will be published elsewhere  \cite{Ya22}.
For the counterexamples concerning polyconvex gradient flows, the question is reduced to a system with two equations; i.e., $m=2.$ (See Section 3.) We present a general  parametrization for some special  $\tau_N$-configurations in $\R^{2\times n}\times \R^{2\times n}$ and show that the graph of $DF$ for a polyconvex function $F$ can support such special $\tau_N$-configurations easily when $n\ge 4.$ (See Section 4.)  The result is also true when $n=2$ (see \cite{Ya20}), but for $n=3$ some more complicated calculations are necessary that cannot be done by hand. Finally we elaborate on the main steps toward   the verification of Condition $(OC)_N$,  including  the dimension $n=2.$ (See Section 5.)   The only difficulties encountered would be the verification of several rational or analytic functions being not identically zero; these verifications may require some necessary calculations.

 \section{The $\tau_N$-configurations and Condition $(OC)_N$} \label{s2}

We first recall the definition of $\tau_N$-configurations introduced in \cite[Definition 3.1]{Ya20}.
 \begin{definition}\label{tau-N}
 Let $N>1.$ An  ordered $N$-tuple $(\xi_1,\xi_2,\dots,\xi_N)$ in $\R^{m\times n}\times \R^{m\times n}$ is called a  {\it $\tau_N$-configuration}  provided that there exist $\rho\in \R^{m\times n}\times \R^{m\times n}$, $\kappa_i>1$ and $\gamma_i =(p_i\otimes  a_i, s_iB_i),$ for $ i=1,\cdots,N,$ such that
\begin{equation}\label{t-N}
  \begin{cases}
\xi_1  =   \rho+\kappa_1\gamma_1,\\
\xi_2   =  \rho+\gamma_1+\kappa_2\gamma_2,\\
\quad \quad  \vdots   \\
\xi_N   =  \rho+\gamma_1+\cdots+\gamma_{N-1}+\kappa_N\gamma_N,
\end{cases}
\end{equation}
where  $p_i \in\R^m$, $a_i  \in \R^n$, $s_i\in\R$ and $B_i\in \R^{m\times n}$ satisfy 
 \begin{equation}\label{sum-0}
\begin{cases}  B_i a_i=0,\;\; (|p_i|+|B_i|) (| a_i|+|s_i|) \ne 0   \quad \forall\,i=1, \cdots,N, \\[1ex]
 \sum\limits_{i=1}^N p_i\otimes  a_i =0, \;\;   \sum\limits_{i=1}^N B_i\otimes  a_i=0,\;\;   \sum\limits_{i=1}^N s_ip_i =0,\;\; \sum\limits_{i=1}^N s_i B_i =0.\end{cases}
\end{equation} 

Let  $\pi_1=\rho,\; \pi_i=\rho+\gamma_1+\cdots+\gamma_{i-1}$ for $i=2,\cdots, N,$ and $\chi_i=\frac{1}{\kappa_i}\in (0,1).$  Then, in terms of $\chi_i$, $\xi_i$ and $\pi_i$, the $\tau_N$-configuration  (\ref{t-N})  is equivalent to   
\begin{equation}\label{t-N-2}
\begin{cases}
\pi_{i+1}-\pi_i=\gamma_i,\\
 \pi_{i+1}=  \chi_i \xi_i+(1-\chi_i)\pi_i  \end{cases} 
 \quad\forall\, i \; \mod N,
 \end{equation}
where $\gamma_i=(p_i\otimes  a_i, s_iB_i)$ with $\{(p_i,a_i,s_i,B_i)\}_{i=1}^N$ satisfying (\ref{sum-0}).  \\
(See Figure \ref{fig0} for an illustration of a $\tau_5$-configuration.)  

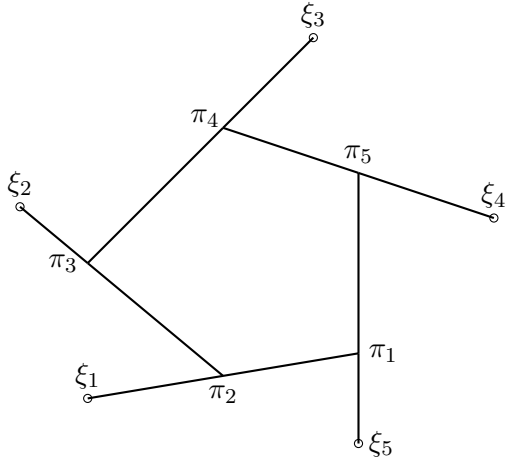
\begin{figure}[ht]
\begin{center}
\begin{tikzpicture}[scale =0.6]
\draw[thick] (-5,-2)--(1,-1);
  \draw(-5,-2) node[above]{$\xi_1$};
   \draw(1,-1) node[right]{$\pi_1$};
     \draw[thick] (1,-1)--(1,3);
  \draw[thick] (1,-1)--(1,-3);
  \draw[thick] (-2,-1.5)--(-5,1);
   \draw[thick] (-6.5,2.25)--(-5,1);
   \draw[thick] (-5,1)--(-2,4);
     \draw[thick] (0,6)--(-2,4);
    \draw[thick] (-2,4)--(1,3);
       \draw[thick] (1,3)--(4,2);
     \draw(1,-3) node[right]{$\xi_5$};
      \draw(1,3) node[above]{$\pi_5$};
        \draw(-5,1) node[left]{$\pi_3$};
         \draw(-2,-1.5) node[below]{$\pi_2$};
               \draw(-2.4,3.8) node[above]{$\pi_4$};
    \draw(4,2) node[above]{$\xi_{4}$};
   \draw(-6.5,2.25) node[above]{$\xi_2$};
     \draw(0,6) node[above]{$\xi_3$};
       \draw (-5,-2) circle (0.09);
        \draw(4,2) circle (0.09);
         \draw(1,-3) circle (0.09);
          \draw(-6.5,2.25) circle (0.09);
           \draw(0,6) circle (0.09);
\end{tikzpicture}
\end{center}
\caption{An illustration of a $\tau_5$-configuration $(\xi_1,\cdots,\xi_5)$}
 \label{fig0}
\end{figure}
 \end{definition}

 We now define the  Condition $(OC)_N$  as a detailed refinement  of Condition (OC)  introduced in \cite[Definition 3.2]{Ya20}.
  
 \begin{definition}\label{OC-N}  Let    $\sigma\colon \R^{m\times n}\to \R^{m\times n}$  and 
$
 \K=\{(A,\sigma(A)):A\in \R^{m\times n}\}
$
 be the graph of $\sigma$ in $\R^{m\times n}\times \R^{m\times n}.$   Given $N>1,$ we say that $\sigma$ satisfies  {\it Condition $(OC)_N$} provided that there exist smooth  functions
\[
\chi_i\colon \bar \B_r(0) \to [\nu_0,\nu_1],\quad \xi_i\colon \bar \B_r(0)   \to \K,\quad  \pi_i\colon \bar \B_r(0) \to \R^{m\times n}\times \R^{m\times n} 
\]
for $i=1,\cdots,N,$ where $\bar \B_r(0)$ is a closed ball in  $\R^{m\times n}\times \R^{m\times n}$ and $0<\nu_0<\nu_1<1$ are some numbers, with the following properties:
 \begin{itemize}
 \item[\bf (P1)] $\pi_1(\rho)=\rho$ for all $\rho\in \bar \B_r(0)$,  $\pi_i(\B_r(0))$ is open  for $1\le i\le N$, and
\[
 \xi^1_i(\bar \B_r(0))\cap   \xi^1_j(\bar \B_r(0))=  \pi^1_i(\bar \B_r(0))\cap  \pi^1_j(\bar \B_r(0))=\emptyset \quad \forall\, i\ne j,
 \]
where $\eta^1\in \R^{m\times n}$ is the projection of  $\eta=(\eta^1,\eta^2)\in\R^{m\times n}\times \R^{m\times n}.$   
\item[\bf (P2)] There exist smooth functions $p_i(\rho)\in \R^m$, $a_i(\rho)\in  \R^n$, $s_i(\rho)\in \R$ and $B_i(\rho)\in \R^{m\times n}$ on $\rho\in\bar \B_r(0),$ for  $i=1,\cdots, N,$ such that  condition (\ref{sum-0}) is satisfied point-wise on $\rho\in\bar \B_r(0)$ and, for $\gamma_i(\rho)= (p_i(\rho)\otimes  a_i(\rho),\, s_i(\rho) B_i(\rho))$ and $\rho\in\bar \B(0),$ it follows that 
\[
\begin{cases}
 \pi_{i+1}(\rho)-\pi_i(\rho)=\gamma_i(\rho),\\
 \pi_{i+1}(\rho)= \chi_i(\rho) \xi_i(\rho)+(1-\chi_i(\rho))\pi_i(\rho)  \end{cases} 
 \; \forall\, i\; \mod N.
 \]

 \item[\bf (P3)] For $0\le \lambda\le 1$ and $1\le i\le N,$ define the subsets in $\R^{m\times n}\times \R^{m\times n}$:
\[
\begin{cases}
S_i(\lambda)   =\{\lambda \xi_i(\rho) +(1-\lambda)\pi_i(\rho): \rho\in \B_r(0)\}, \\
K(\lambda)  =  \cup_{i=1}^N S_i(\lambda),  \\
 \Sigma(0)=K(0), \quad 
 \Sigma(\lambda) = \cup_{0\le \lambda'<\lambda} K(\lambda') \quad\forall\, 0<\lambda\le 1.
 \end{cases}
\]
 Then  $\Sigma(\lambda)$ is open for all $\lambda\in [0,1];$ 
moreover, there is a number  $ \delta_1\in (0,1)$ such that
 $\{S_i(\lambda)\}_{i=1}^N $ is a family of disjoint  open sets   for each $\lambda\in [\delta_1,1).$  
\end{itemize}  
\end{definition}

We state the following  theorem, which guarantees the existence of  nowhere-$C^1$   Lipschitz  weak solutions to the initial boundary value problem for system (\ref{DE}) under the assumption of Condition $(OC)_N.$ The proof of this  theorem will be given elsewhere; see \cite{Ya22} for  the case $m=n=2.$   
 
 \begin{theorem}\label{mainthm1}  Let  $\sigma$ be locally Lipschitz and satisfy  Condition $(OC)_N$  for some $N>1.$ Then for any $(\bar u, \bar v)\in C^1(\bar \Omega_T;\R^m\times\R^{m\times n})$ satisfying  
\begin{equation}\label{subs}
\bar u=\dv \bar v, \quad (D\bar u,\partial_t \bar v)\in \Sigma(\bar \lambda) \;\;\mbox{  
 on $\bar \Omega_T,$}
 \end{equation}
where $0< \bar \lambda<1,$ and  $\delta>0,$  the {\em full  Dirichlet   problem}
\begin{equation}\label{ibvp-5}
\begin{cases}  \partial_t u=\dv \sigma(Du )   \;\; \mbox{\rm on $\Omega_T,$}
\\
 u  |_{\partial\Omega_T}=\bar u 
\end{cases}
\end{equation}  possesses a  Lipschitz weak solution   $u$ with
$\|u-\bar u\|_{L^\infty}+\|\partial_t u-\partial_t \bar u\|_{L^\infty} <\delta$ such that $Du$ is not  essentially  continuous at any point of $\Omega_T.$
 \end{theorem}


 \section{Counterexamples for polyconvex gradient flows}   
 
 We consider the following problem concerning nowhere-$C^1$ Lipschitz solutions to the initial boundary value problem of polyconvex gradient flows.  
 
 \begin{problem}\label{mainthm0}
Given $m,n\ge 2$, find a function  $F\colon \R^{m\times n}\to \R$ of the form
\begin{equation}\label{poly}
F(A)=\frac{\nu}{2}|A|^2 +G(\Gamma (A)),
\end{equation}
where $\nu>0$, $\Gamma (A) \in \R^r$ is an ordered $r$-tuple  of certain subdeterminants of $A,$ and $G\colon\R^{r} \to\R$ is a smooth convex function,  such that  the initial boundary value problem 
\begin{equation}\label{ibvp}
\begin{cases}  \partial_t  u=\dv DF(D u )   \;\; \mbox{in $\Omega_T,$}
\\
 u  |_{\partial' \Omega_T}=0,  
 \end{cases}
\end{equation}  
where $\partial'\Omega_T=(\Omega\times \{0\})\cup (\partial \Omega \times (0,T))$ is  the parabolic boundary of $\Omega_T,$ 
 possesses  infinitely many  Lipschitz weak solutions for which $Du$ is nowhere essentially   continuous  in $\Omega_T.$  
\end{problem}

Problem \ref{mainthm0} can be reduced to the case   $m=2$ as follows. 
 Let  $m>2.$ Consider 
\[
F(A)=\frac{\nu}{2}|A|^2 +G(A_1, J(A_1)) \quad \forall\,A\in \R^{m\times n},
\]
where $A_1\in \R^{2\times n}$ is the submatrix of first two rows of $A$ and  $J(A_1) \in \R^r$ is  an ordered $r$-tuple of  certain $2\times 2$ subdeterminants  of $A_1.$  Then
\[
DF(A)=\nu A +\begin{bmatrix} G_{A_1}(\tilde A_1)\\O\end{bmatrix} + \begin{bmatrix} G_{J}(\tilde A_1)DJ(A_1)
\\O\end{bmatrix},
\]
where $\tilde A_1=(A_1, J(A_1)).$ Let 
\[
\tilde F(A_1)=\frac{\nu}{2}|A_1|^2 +G(A_1, J(A_1)) \quad \forall\,A_1\in \R^{2\times n}.
\]
Consider functions  $u=(\tilde u,0,\cdots,0),$ where $\tilde u=(u^1,u^2)\colon \Omega_T\to \R^2.$ We have 
\[
\partial_t u=\begin{bmatrix} \partial_t\tilde u\\0\end{bmatrix},\quad Du=\begin{bmatrix} D \tilde u\\O\end{bmatrix},\quad DF(Du)=\begin{bmatrix} D\tilde F(D\tilde u)\\O\end{bmatrix}.
\]
For such functions, problem (\ref{ibvp}) is   equivalent to the zero initial boundary value problem for system 
$\partial_t \tilde u =\dv D\tilde F(D\tilde u).$ This reduces the question to the case  $m=2.$
However, there seems no way to reduce the dimension $n$. 

In what follows, we  consider functions $F\colon \R^{2\times n}\to\R$ of the form
\begin{equation}\label{F0}
F(A)=\frac{\epsilon}{2}|A|^2+G(A,J(A)) \quad\forall \, A\in \R^{2\times n},
\end{equation}
where $\epsilon>0$, $J(A)\in \R^d$  is an ordered $d$-tuple of {\it all} $2\times 2$ subdeterminants  of $A,$ with $d= {n(n-1)}/{2}$, and $G\colon \R^{2\times n}\times \R^d\to\R$ is a smooth function. Thus 
\[
DF(A)=\epsilon A+ G_A(\tilde A)+G_J(\tilde A) DJ(A),
\]
where  $\tilde A= (A,J(A)).$

By virtue of Theorem \ref{mainthm1}, the resolution of Problem \ref{mainthm0} relies on constructing  functions $F\colon \R^{2\times n}\to\R$ of the form
(\ref{F0}) with smooth convex $G\colon \R^{2\times n}\times \R^d\to\R$ such that $\sigma=DF$ satisfies Condition $(OC)_N.$   
 
 \section{Polyconvex functions  supporting special $\tau_N$-configurations} 
 In this section, we construct a polyconvex  function  $F_0\colon \R^{2\times n}\to\R$ of the form
(\ref{F0}) with  smooth convex $G\colon \R^{2\times n}\times \R^d\to\R$ such that the graph $\K_{F_0}$ of $\sigma=DF_0$ contains a special $\tau_N$-configuration for some $N>1.$

 \subsection{Parametrizing some special $\tau_N$-configurations}  We consider some special  $\tau_N$-configurations  in $\R^{2\times n}\times \R^{2\times n}.$ 
 
 First, we  modify the second condition in the first line of condition  (\ref{sum-0}) by considering the elements $\{(p_i, a_i,B_i, s_i)\}_{i=1}^N$ in $\R^2\times \R^n\times \R^{2\times n}\times \R$  satisfying   
 \begin{equation}\label{sum-1}
\begin{cases}  B_i a_i=0,\;\;  |p_i| | a_i| \ne 0   \quad \forall\,i=1, \cdots,N, \\[1ex]
 \sum\limits_{i=1}^N p_i\otimes  a_i =0, \;\;   \sum\limits_{i=1}^N B_i\otimes  a_i=0,\\
   \sum\limits_{i=1}^N s_ip_i =0,\;\; \sum\limits_{i=1}^N s_i B_i =0.\end{cases}
\end{equation} 
Second, we observe  that  if $s_i=a_i\cdot q$ for all $i=1, \cdots,N,$ where $q\in\R^n$ is a fixed nonzero vector,  then the third-lined conditions   of (\ref{sum-1}) will follow automatically   from the second-lined conditions. So, we will focus on such special $\tau_N$-configurations.

To explicitly parametrize the first-lined and second-lined conditions  of (\ref{sum-1}), we define for all $r=1,2,\cdots,n$ and $\x, \, \y,\,\z\in\R^{n-1},$  
\begin{equation}
\begin{cases}
   \alpha_r(\x)   =(x_1,\cdots,x_{r-1},1,x_r,\cdots,x_{n-1})^T\in \R^n,\\
  b_r(\x,\y)   =(y_1,\cdots,y_{r-1},\,-(\x\cdot \y)\,,y_r,\cdots,y_{n-1})\in (\R^n)^T,\\
  \beta_r(\x,\y,\z)   =\begin{bmatrix} b_r(\x,\y)\\b_r(\x,\z)\end{bmatrix},
\end{cases}
\end{equation}
 where $\x=(x_1,\cdots,x_{n-1}),$ etc. Then 
 \[
 \beta_r(\x,\y,\z)\alpha_r(\x)=0 \quad \forall\, \x,\y,\z\in\R^{n-1},\;\; r=1, \cdots,n. 
 \]

Assume $N\ge n+3$  and let $r_1,\cdots,r_N\in \{1, \cdots,n\}$ be such that 
\[
r_i=i \, \mod n\quad \forall\, i=1,\cdots,N.
\]
Given $\x_i,\y_i,\z_i\in\R^{n-1}$ and $p_i\in \R^2$ for  $i=1, \cdots,N,$  let
   \[
a_i=\alpha_{r_i}(\x_i), \quad B_i=\beta_{r_i}(\x_i,\y_i,\z_i).
\]
 For such $a_i$ and $B_i$, we write the second-lined conditions of (\ref{sum-1})  as  
 \begin{equation}\label{sum-2}
 \sum_{i=1}^n p_i\otimes  a_i =- \sum_{j=n+1}^N p_j\otimes  a_j, \quad \sum_{i=1}^{n+1} B_i\otimes  a_i=-\sum_{j=n+2}^N B_j\otimes  a_j.
\end{equation}
Define 
\begin{equation}\label{det-1}
\Delta (\x_1,\cdots,\x_n)=\det [\alpha_1(\x_1)\; \alpha_2(\x_2)\; \cdots\; \alpha_n(\x_n)],
\end{equation}
which is a nonzero polynomial in $(\x_1,\cdots,\x_n).$
If $\Delta (\x_1,\cdots,\x_n)\ne 0,$ then from the first equation of (\ref{sum-2}), we can solve for $p_1,\cdots,p_n$ to obtain
\begin{equation}\label{p-1}
p_i=\sum_{j=n+1}^N \frac{S_{ij}(\x_1,\cdots,\x_N)}{\Delta(\x_1,\cdots,\x_n)}p_j \quad \forall\,i=1, \cdots,n,
\end{equation}
where all the coefficients $S_{ij}(\x_1,\cdots,\x_N)$  are  polynomials of $(\x_1,\cdots,\x_N).$ 

Since $B_i\otimes a_i=\begin{pmatrix} b_{r_i}(\x_i,\y_i)\otimes \alpha_{r_i}(\x_i)\\ b_{r_i}(\x_i,\z_i)\otimes \alpha_{r_i}(\x_i)\end{pmatrix} $ 
and  $b_{r_i}(\x_i,\y_i) \alpha_{r_i}(\x_i)=0$, we have  
\begin{equation}\label{trace}
(b_{r_i}(\x_i,\y_i)\otimes \alpha_{r_i}(\x_i))_{11}=-\sum_{j=2}^n (b_{r_i}(\x_i,\y_i)\otimes \alpha_{r_i}(\x_i))_{jj} 
\end{equation}
for all $ i=1,\cdots,N;$  the same  holds for $(b_{r_i}(\x_i,\z_i)\otimes \alpha_{r_i}(\x_i))_{11}.$ The second equation of (\ref{sum-2}) is equivalent to two sets of equations:
\begin{equation}
  \sum_{i=1}^{n+1} b_{r_i}(\x_i,\y_i)\otimes \alpha_{r_i}(\x_i) =-\sum_{j=n+2}^N b_{r_j}(\x_j,\y_j)\otimes \alpha_{r_j}(\x_j), \label{eq-y}
  \end{equation}
  \begin{equation}
 \sum_{i=1}^{n+1} b_{r_i}(\x_i,\z_i)\otimes \alpha_{r_i}(\x_i) =-\sum_{j=n+2}^N b_{r_j}(\x_j,\z_j)\otimes \alpha_{r_j}(\x_j).\label{eq-z}
\end{equation}
Because of  (\ref{trace}), both (\ref{eq-y}) and (\ref{eq-z}) are the same linear system of $(n^2-1)$ equations for $(n+1)(n-1)=(n^2-1)$ variables $\{\y_1,\cdots,\y_{n+1}\}$ and $\{\z_1,\cdots,\z_{n+1}\},$  respectively.   If we order $\{\y_1,\cdots,\y_{n+1}\}$ as 
\[
\y'=(y_{11}, \cdots,y_{1(n-1)};\, y_{21}, \cdots,y_{2(n-1)};\, \cdots; \, y_{(n+1)1}, \cdots,y_{(n+1)(n-1)})
\]
as a vector in $ \R^{n^2-1},$ then the coefficient matrix $A$ of  system (\ref{eq-y})  is a  square   $(n^2-1)\times (n^2-1)$ matrix  $ A=A(\x_1,\x_2,\cdots,\x_{n+1})$ whose  entries are all polynomials of $(\x_1,\x_2,\cdots,\x_{n+1})$ of degree at most  2.  Define
\begin{equation}\label{det-2}
T(\x_1,\cdots,\x_{n+1})=\det A (\x_1,\cdots,\x_{n+1}),
\end{equation}
which is a nonzero polynomial in $(\x_1,\cdots,\x_{n+1}).$ If $T(\x_1,\cdots,\x_{n+1})\ne 0$, then we can solve $\y'=(\y_1,\cdots,\y_{n+1})$ from (\ref{eq-y}) to obtain
\begin{equation}\label{def-y}
\y'= \frac{L (\x_1,\cdots,\x_N)(\y_{n+2},\cdots,\y_N)}{T(\x_1,\cdots,\x_{n+1})},
\end{equation}
where $L (\x_1,\cdots,\x_N)$ is a linear  mapping on $ (\y_{n+2},\cdots,\y_N)$  with all coefficients being polynomials of $(\x_1,\cdots,\x_N).$ In  the same way, we have
\begin{equation}\label{def-z}
\z'= \frac{L (\x_1,\cdots,\x_N)(\z_{n+2},\cdots,\z_N)}{T(\x_1,\cdots,\x_{n+1})}.
\end{equation}
 
Let  $\mathcal V$ be the set of elements $(P,X)$ with $P =(p_{n+1},\cdots,p_N)\in (\R^2)^{N-n}$ and $X=(\x_1,\cdots,\x_N)\in (\R^{n-1})^N$  satisfying  
\begin{equation}\label{set-V}
\begin{cases} \Delta(\x_1,\cdots,\x_n) \ne 0,\\
T(\x_1,\cdots,\x_{n+1}) \ne 0,\\
p_j \ne  0 \quad\forall\, j=n+1,\cdots,N,\\
\sum\limits_{j=n+1}^N  S_{ij}(\x_1,\cdots,\x_N) p_j \ne  0 \quad  \forall\,i=1,\cdots,n,\end{cases}
\end{equation}
where $S_{ij}(\x_1,\cdots,\x_N)$  are the polynomials  defined in (\ref{p-1}). 
Since all $\Delta$, $T$, $p_j$ and $
\sum_{j=n+1}^N  S_{ij} p_j$ are nonzero polynomials in $(P,X)$, it follows  that $\mathcal V$ is a nonempty open dense set in $(\R^2)^{N-n}\times (\R^{n-1})^N.$ 

Define $\U$ to be the set of all elements 
\begin{equation}\label{U}
U= (P,\, X, \,   Y,\, Z,\, b,\, \kappa_1,\cdots,\kappa_{N}),
\end{equation}
 where $ (P,X) \in \mathcal V$,   $Y=(\y_{n+2},\cdots,\y_N)$ and $Z=(\z_{n+2},\cdots,\z_N)$ with  $\y_j,\z_j\in \R^{n-1}$ for $n+2\le j\le N$, $b\in \R^{n-2},$  and  $\kappa_i>1$ 
for $1\le i \le N.$  Then $\U$ is an open set in $\R^D,$ where the total dimension  $D$ equals
 \[
 2(N-n) +N(n-1)+2(N-n-1)(n-1)+(n-2)+N=3nN-2n^2-n.
  \]  
    
  For each $U\in\U$  given by (\ref{U}), we define  $p_i=p_i(U)=p_i(P,X)$ for $i=1, \cdots,n$ by formula (\ref{p-1}) and define
  $\y'=\y'(U)=\y'(X,Y)$ and $\z'=\z'(U)=\z'(X,Z)$ by formulas (\ref{def-y}) and (\ref{def-z}). In this way, we obtain 
\begin{equation}\label{def-p}
  p_i=p_i(U)\in\R^2\setminus\{ 0\}, \; \;  \y_i=\y_i(U)\in \R^{n-1},\;\; \z_i= \z_i(U)\in \R^{n-1}
\end{equation} 
  for all $i=1, \cdots,N.$  Also define $q(U)=q(b)=(1,1,b_1,\cdots,b_{n-2})^T\in\R^n$ if $b=(b_1,\cdots,b_{n-2})^T\in\R^{n-2}.$ Finally we define 
\begin{equation}\label{aBsg}
\begin{cases}    a_i(U)=\alpha_{r_i}(\x_i),\\ 
  B_i(U)=\beta_{r_i}(\x_i,\y_i(U),\z_i(U)),\\
  s_i(U)=a_i(U)\cdot q(b),\\
  \gamma_i(U)   = (p_i(U)\otimes a_i(U), \, s_i(U)B_i(U)), \end{cases} \; \forall\, i=1, \cdots,N.
\end{equation} 
Note that all these functions are  independent of $(\kappa_1,\cdots,\kappa_N).$ 

 Let 
 \begin{equation}\label{eta}
\begin{cases}  
 \eta_1(U)    =\kappa_1 \gamma_1(U),\\
 \eta_2(U)    = \gamma_1(U)+\kappa_2 \gamma_2(U),\\
 \qquad \vdots\\
 \eta_N(U)    = \gamma_1(U)+\gamma_2(U)+\cdots+ \gamma_{N-1}(U)+\kappa_N \gamma_N(U).\end{cases}
\end{equation}

 We summarize  the constructions above in the following result.
 
 \begin{proposition}\label{p-t} For each $U\in \U$, the $N$-tuple $(\eta_1(U),\eta_2(U),\cdots,\eta_N(U))$ is a  $\tau_N$-configuration.  Let $\eta_i(U)=(\eta_i^1(U),\eta_i^2(U))\in \R^{2\times n}\times \R^{2\times n}.$ Then each $\eta_i^1(U)$ is  linear  in  $P$ with coefficients  rational functions of  $(X,  \kappa_1,\cdots,\kappa_N),$ and each  $\eta_i^2(U)$ is  linear   in  $(Y,Z)$ with coefficients  rational functions of 
 $(X, b,\kappa_1,\cdots,\kappa_N).$  
 \end{proposition}
 
The following result will be useful later.

\begin{lemma}\label{rk5} Let $\zeta_i(U)=\kappa_i\gamma_i(U)$ for $i=1,\cdots,N$. Then  
\[
\rank D\zeta_i(U) \le 3n-1.
\]
\end{lemma} 
\begin{proof} Write $\zeta_i(U)\in \R^{2\times n}\times \R^{2\times n}$ as a column vector in $ \R^{4n}$:
\[
\zeta_i(U)=\begin{bmatrix} f(U)\alpha_{r_i}(\x_i)\\[1ex]
g(U)\alpha_{r_i}(\x_i)\\[1ex]
h(U)b^T_{r_i}(\x_i,\y_i(U))\\[1ex]
h(U)b^T_{r_i}(\x_i,\z_i(U))\end{bmatrix},
\]
where $f(U)=\kappa_i p_i^1(U),g(U)=\kappa_i p_i^2(U)$ and $h(U)=\kappa_i s_i(U).$   Since $p_i(U)\ne 0,$ without loss of generality,  we  assume $f(U)\ne 0.$ Note that
\[
D( f(U)\alpha^j_{r_i}(\x_i))=fD\alpha^j_{r_i}+\alpha^j_{r_i} Df \quad \forall\,j\ne r_i,
\]
\[
D( f(U)\alpha^{r_i}_{r_i}(\x_i))=  Df \quad \mbox{(since  $\alpha_{r_i}^{r_i}(\x_i)=1$).}
\]
After elementary row operations, we have the following row equivalence: 
\[
D[ f(U)\alpha_{r_i}(\x_i)] \sim  \begin{bmatrix} f(U)D\alpha'_{r_i}(\x_i)\\
Df(U)\end{bmatrix}\sim  \begin{bmatrix} f(U)D \x_i \\
Df(U)\end{bmatrix},
\]
where $\alpha_{r_i}'=(\alpha_{r_i}^j)_{j\ne r_i}.$ Similarly,
\[
D[ g(U)\alpha_{r_i}(\x_i)] \sim  \begin{bmatrix} g(U)D \x_i\\
Dg(U)\end{bmatrix}
\]
and thus 
\[
D\begin{bmatrix}  f(U)\alpha_{r_i}(\x_i) \\
g(U)\alpha_{r_i}(\x_i) \end{bmatrix}
\sim  \begin{bmatrix} f(U)D \x_i\\
Df(U)\\
g(U)D \x_i\\
Dg(U) \end{bmatrix}\sim  \begin{bmatrix} f(U)D \x_i \\
Df(U) \\ 
Dg(U) \end{bmatrix}.
\]
Note that by interchanging rows,  we have $b'_{r_i}(\x,\y)\sim \y$,  where $b'_{r_i}=(b_{r_i}^j)_{j\ne r_i}$   and $b_{r_i}^{r_i}(\x,\y)=-(\x\cdot\y).$ Thus
\[
D[h(U)b'_{r_i}(\x_i,\y_i(U))]\sim h(U)D \y_i(U)  +  \y_i(U)\otimes Dh(U),
\] 
\[
D[h(U)b^{r_i}_{r_i}(\x_i,\y_i(U))]=-h(U)D(\x_i\cdot \y_i(U)) - (\x_i\cdot \y_i(U) ) Dh(U).
\]
After elementary row operations and using $D(\x_i\cdot \y_i(U))=\x_i \cdot D\y_i(U)+\y_i(U)D\x_i,$ we have
\[
D[h(U)b^T_{r_i}(\x_i,\y_i(U))]\sim  \begin{bmatrix} h(U)D \y_i(U)  +  \y_i(U)\otimes Dh(U)\\
h(U)\y_i(U)\cdot D \x_i  \end{bmatrix}.
\]
(There is no row cancellation here!) Thus
\[
D\begin{bmatrix} h(U)b^T_{r_i}(\x_i,\y_i(U))  \\[1ex]  h(U)b^T_{r_i}(\x_i,\z_i(U))\end{bmatrix}
\sim  \begin{bmatrix} h(U)D \y_i(U)  +  \y_i(U)\otimes Dh(U)\\
h(U)\y_i(U)\cdot D \x_i \\
h(U)D \z_i(U)  +  \z_i(U)\otimes Dh(U)\\
h(U)\z_i(U)\cdot D \x_i \end{bmatrix}.
\]
Since both $h(U)\y_i(U)\cdot D \x_i $ and $h(U)\z_i(U)\cdot D \x_i$ can be cancelled by $f(U)D\x_i$, we finally have
\[
D\zeta_i(U)\sim \begin{bmatrix} f(U)D\x_i\\
Df(U) \\ 
Dg(U) \\
h(U)D \y_i(U)  +  \y_i(U)\otimes Dh(U)\\
h(U)D \z_i(U)  +  \z_i(U)\otimes Dh(U)\end{bmatrix},
\]
which has $(3n-1)$ rows. This completes the proof.
\end{proof}

\subsection{Embedding $\tau_N$-configurations onto the graph of $DF$} 

 Our goal is  to embed  a parameterized family of $\tau_N$-configurations of the form 
 \[
  (\rho+\eta_1(U), \rho+ \eta_2(U),\cdots,\rho + \eta_N(U)),
 \]
 where $\rho\in\R^{2\times n}\times \R^{2\times n}$ and $U\in \U,$ onto the graph  $\K_F$ of $DF$ for  some function $F$ of the form (\ref{F0}) with a convex $G.$ To this end, let  $\Phi(A,B)=DF(A)-B$ and we try to solve
 \[
\Psi(\rho,U):= (\Phi(\rho+\eta_1(U)), \cdots, \Phi(\rho+\eta_N(U)))=0\in\R^{2nN}.
 \]
If we would like to apply the implicit function theorem to solve for $U$ in terms of  $\rho,$ then it is necessary that 
\begin{equation}\label{N}
D=\dim \U= 2nN; \quad \mbox{thus,} \;\; N=2n+1.
\end{equation}

In what follows, let $N=2n+1$   and  we first try to embed one  $\tau_N$-configuration
$(\eta_1(U),   \eta_2(U),\cdots, \eta_N(U))$  onto  the graph $\K_F$ for some $U\in\U;$  that is, 
\begin{equation}\label{on-K}
\Phi(\eta_i(U))=0\quad \forall\, i=1,\cdots,N.
\end{equation}
 Since
$
\Phi(A,B)=\epsilon A+G_A(\tilde A) +G_J(\tilde A)DJ(A)-B,
$
where $\tilde A=(A,J(A))\in \R^{2\times n}\times \R^d,$ condition (\ref{on-K})  is equivalent to  that for  all $i=1,\cdots,N,$
\begin{equation}\label{emb-1}
G_A(\tilde  \eta_i^1(U) )=\eta_i^2(U)-\epsilon \eta_i^1(U)-G_J(\tilde  \eta_i^1(U)  ) DJ(\eta_i^1(U)),
\end{equation}
where $\tilde  \eta_i^1(U) =( \eta_i^1(U), J(\eta_i^1(U))).$ 

As in \cite{Sz04,Ya20}, we can construct  a smooth convex function $G\colon \R^{2\times n}\times \R^d\to\R$  such that 
\[
G (\tilde \eta_i^1(U) ) =c_i,\;\; G_J(\tilde \eta_i^1(U)  )=d_i,\;\; G_A(\tilde \eta_i^1(U) )=Q_i\quad\forall\,i=1,\cdots,N
\]
 {\it provided} that the quantities $\{c_i,d_i,Q_i\}_{i=1}^N$ satisfy 
\begin{equation}\label{cx-0}
c_j-c_i> \langle Q_i, \eta_j^1(U) -\eta_i^1(U) \rangle +d_i\cdot [J(\eta_j^1(U))-J(\eta_i^1(U))]
\end{equation}
 for all $1\le i\ne j\le N.$ 
 
If $Q_i= G_A(\tilde \eta_i^1(U) )$ is given by (\ref{emb-1}) with $G_J(\tilde \eta_i^1(U)  )=d_i,$ then condition (\ref{cx-0}) will be  satisfied for all sufficiently small $\epsilon>0$ {\it provided} that the following condition is satisfied:
\begin{equation}\label{emb-2}
 c_j-c_i>   \langle \eta_i^2(U), \eta_j^1(U) -\eta_i^1(U) \rangle   +d_i \cdot J(\eta_j^1(U)-\eta_i^1(U) ) \quad \forall\, i\ne j,
\end{equation}
where we have used the identity: $
J(A)-J(B) -DJ(B )(A -B)=J(A-B).$

 From the specific dependence  of $\eta_i^1(U)$ and $\eta_i^2(U)$ on $U$ as described in Proposition \ref{p-t}, we see that
\[
  \langle \eta_i^2(U), \eta_j^1(U) -\eta_i^1(U) \rangle =(Y,Z)\cdot \mathcal S_{ji}(P,U'),
\]
where  $U' =(X,b,\kappa_1,\cdots,\kappa_N)$ and $\mathcal S_{ji}(P,U')\in \R^{2[(N -n-1)(n-1)] }
 $ are linear in $P$ with coefficients being rational functions of $U'.$  Also write $J(\eta_j^1(U)-\eta_i^1(U))=J_{ji}(P,U')$ as a {\it quadratic} function in $P$ with  coefficients being rational functions of $U'.$  In this way, condition (\ref{emb-2}) becomes 
\begin{equation}\label{emb-3}
  c_i-c_j+  (Y,Z)\cdot \mathcal S_{ji}(P,U') +d_i \cdot J_{ji}(P,U') <0 \quad \forall\,   i\ne j,
\end{equation} 
which  consists of  $N(N-1)$ homogeneous linear inequalities  on the variables $\{(c_i,d_i)\}_{i=1}^N$ and  $(Y,Z).$  The total number of the variables is
\[
N(1+d)+2(N-n-1)(n-1).
\]
For example, if $n=2$ then  (\ref{emb-3}) is a $20\times 14$ system of homogeneous  linear inequalities, and if $n=3$ then  (\ref{emb-3})  is a $42\times 40$ system, both being overdetermined; however, when $n\ge 4$ the condition (\ref{emb-3})  becomes underdetermined.  In fact, we have the following result.

\begin{proposition}\label{n4} Let $n\ge 4$ and $N=2n+1.$ Given any $\{c_i\}_{i=1}^N$, there exist $d_1, \cdots, d_N$ in $\R^d$  such that condition (\ref{emb-2}) is  satisfied at some point $U_0\in \U.$  In this case, (\ref{emb-2}) is also satisfied for all $U$ in a neighborhood  $ \mathcal N(U_0)$ of $U_0$ in $\U$ by  the same $\{c_i\}_{i=1}^N$ and $\{d_i\}_{i=1}^N.$
\end{proposition}
\begin{proof}  Given  $\{c_i\}_{i=1}^N$, we solve the  linear system: 
\begin{equation}\label{emb-4}
 (Y,Z)\cdot \mathcal S_{ji}(P,U') +d_i \cdot J_{ji}(P,U' )=-1-c_i+c_j \quad \forall\,  i\ne j  
\end{equation} 
 for unknown variables  $d_1,\cdots,d_N$ and $(Y,Z).$ The total number of the unknown variables  is
\[
Nd+2(N-n-1)(n-1),
\]
 which is greater than the number $N(N-1)$ of equations if $n\ge 4;$ thus system  (\ref{emb-4}) is underdetermined. The coefficient matrix  $\mathcal M(P,U' )$  of   (\ref{emb-4}) 
has all entries being polynomials in $P$  and rational functions in $U',$ and its rows have some special forms which ensure that these rows cannot be linearly dependent for all $(P,U').$ Thus $\mathcal M(P,U' )$ will be of the full rank $N(N-1)$  for some $(P_0,U_0')$ with $(P_0,X_0) \in\mathcal V.$ Fix any such  $(P_0,U_0')$ and  the system (\ref{emb-4}) will have  solutions $\{d_i\}_{i=1}^N$ and $(Y_0,Z_0)$ that are also rational functions of $(\{c_i\}_{i=1}^N, P_0,U_0').$ Hence (\ref{emb-3}) is  satisfied by  $\{c_i\}_{i=1}^N$, $\{d_i\}_{i=1}^N$ and $(Y_0,Z_0)$ at $(P_0,U_0')$. With $(P_0,U_0',Y_0,Z_0)$ we obtain $U_0\in\U$ such that condition (\ref{emb-2}) is  satisfied at $U_0$ by  $\{c_i\}_{i=1}^N$ and $\{d_i\}_{i=1}^N.$ Clearly,  (\ref{emb-2}) is also satisfied at all $U$ in a neighborhood $\mathcal N(U_0)\subset\U $ by the same $\{c_i\}_{i=1}^N$ and $\{d_i\}_{i=1}^N.$ 
\end{proof}

\begin{remark} When $n=2$ and $N=5,$ it has been proved in \cite{Ya20} that condition (\ref{emb-2}) is  satisfied by some $\{(c_i,d_j)\}_{i=1}^5$ at some point $U_0\in \U;$  however, the proof  requires  lots of  complicated calculations using  MATLAB.  We believe that the condition (\ref{emb-2})  should be feasible  for $n=3$ as well, but, like in the case $n=2$, a proof would  involve a considerable number of necessary complicated calculations that cannot be done easily by hand. 
\end{remark}

We summarize the constructions above in the following result.

\begin{proposition}\label{pro-F0}  Suppose that condition (\ref{emb-2}) is  satisfied by $\{(c_i,d_i)\}_{i=1}^N$  at some $U_0\in\U.$ Then for some $\epsilon_0>0$ and all $\epsilon\in (0,\epsilon_0)$ there exists  a smooth convex function  $G\colon \R^{2\times n}\times \R^d\to\R$ such that $G(\tilde \eta_i^1(U_0))=c_i$,  $G_J(\tilde \eta_i^1(U_0))=d_i$ and $G_A(\tilde \eta_i^1(U_0))$ is given by (\ref{emb-1}) for $i=1,\cdots,N.$   Let
\[
F_0(A)=\frac{\epsilon}{2} |A|^2+G(A,J(A)).
\]
Then $\eta_i(U_0) \in \K_{F_0}$ and $\eta^1_i(U_0)\ne \eta_j^1(U_0)$ for all $1\le i\ne j \le N.$  
\end{proposition}
 
 \section{Compatibility of polyconvexity and Condition $(OC)_N.$} 
 
 In this section, we  elaborate on the construction of  function $F\colon \R^{2\times n}\to\R$ of the form
(\ref{F0}) with smooth convex $G\colon \R^{2\times n}\times \R^d\to\R$ such that $\sigma=DF$ satisfies Condition $(OC)_N.$   

In what follows, let $F_0$ be the function defined in Proposition \ref{pro-F0}. Note that $F_0$ depends on $U_0\in \U.$

\subsection{Perturbing the function $F_0$}  
Let $B_1(0)$ be the unit ball in $\R^{2\times n}$ and  $\zeta\in C^\infty_c(B_1(0);[0,1])$ be such that
 $\zeta(0)=1.$ Given $r>0$ and tensor $H=(H^{pqij})$ with $H^{pqij}=H^{ijpq}\in\R$ for all $i, p=1,2$ and $j,q=1, \cdots,n$, we define the function
\[
V_{H,r}(A)= \frac12 \zeta(A/r) \sum_{i,p=1}^2\sum_{j,q=1}^n  H^{ijpq}a_{ij}a_{pq}\quad \forall\, A=(a_{ij})\in\R^{2\times n}. 
\]
  Then $V_{H,r}\in C^\infty_c(\R^{2\times n})$ has support in $B_r(0),$ and  
\begin{equation}\label{cut-off}
\begin{cases}
V_{H,r}(0)=0,\quad 
DV_{H,r}(0)=0,\quad
D^2 V_{H,r}(0)=H,\\
|D^2 V_{H,r}(A)|\le C_0 |H| \quad \forall\, A \in\R^{2\times n},
\end{cases}
\end{equation}
where $C_0>0$ is a constant independent of $H,r.$ 
Let 
\[
r_0=\min_{i\ne j} \left \{  |\eta^1_i(U_0)-\eta_j^1(U_0)| \right \}>0
\]
 and define
\begin{equation}\label{def-F}
F(A)=F_0(A)+\sum_{j=1}^N V_{\tilde H_j,r_0}(A-\eta_j^1(U_0)) \quad \forall\, A \in\R^{2\times n},
\end{equation}
where $\tilde H_1,\cdots,\tilde H_N$  are to be chosen later.    Then
\begin{equation}\label{prop-F}
\begin{cases} 
DF(\eta_j^1(U_0))=DF_0(\eta_j^1(U_0)),\\
D^2 F(\eta_j^1(U_0))=D^2 F_0(\eta_j^1(U_0))+\tilde H_j \quad \forall\, 1\le j\le N. \end{cases}  
\end{equation}
Thus  $\eta_j(U_0) \in \K_F$ for all $j=1,\cdots,N.$ 

\begin{lemma} The function $ g(A)=\frac{\epsilon}{4}|A|^2+\sum\limits_{j=1}^N V_{\tilde H_j,r_0}(A-\eta_j^1(U_0))$   is convex provided that
\begin{equation}\label{def-F2}
\sum_{j=1}^N |\tilde H_j|<\frac{\epsilon}{2C_0}, \;\; \mbox{with $C_0$ being the constant in (\ref{cut-off}),}
\end{equation}
In this case, the function $F$ defined by  (\ref{def-F})  is of the form 
\[
F(A)=\frac{\epsilon}{4}|A|^2 + \tilde G(A, J(A)),
\]
where $\tilde G(A, J)= g(A)+G(A,J)$ is  smooth and convex on $\R^{2\times n}\times \R^d.$ 
 \end{lemma}

Given any $(H_1^0,\cdots,H_N^0)$, let  $F$ be the function defined by  (\ref{def-F}) with 
\begin{equation}\label{def-tilde-H}
 \tilde H_j=H_j^0-D^2 F_0(\eta_j^1(U_0)) \quad\forall\,j=1,\cdots,N,
\end{equation}
and define 
\[
\Phi(\xi)=  DF(A)-B\quad \forall\, \xi=(A,B)\in \R^{2\times n}\times\R^{2\times n}.
\]
 Note that,  in terms of the Jacobian matrix, we have
\begin{equation}\label{D-phi}
D \Phi(\xi)=E+\mathcal L(D^2F(A)) \quad \forall\, \xi=(A,B),
\end{equation}
where $E$ is a constant matrix and $\mathcal L(H)$ is a matrix that is linear in $H$ with constant coefficients.  
 Let 
\[
\Psi(\rho,U)=(\Phi(\rho+\eta_1(U)), \cdots, \Phi(\rho+\eta_N(U)))
\]
for all $\rho\in \R^{2\times n}\times \R^{2\times n}$ and $U\in\U.$  Then
$
 \Psi(0,U_0)=0.
$

  The partial derivative  $\frac{\partial\Psi}{\partial U}(\rho,U),$   viewed as a $2nN\times 2nN$ matrix, has all its entries being  the set  
$
\{D\Phi(\rho+\eta_1(U)) D\eta_1(U),\cdots,D\Phi(\rho+\eta_N(U))D\eta_N(U) \}.
$ In particular, all entries of the matrix $\frac{\partial\Psi }{\partial U}(0,U_0)$ are affine in $(H_1^0,\cdots,H_N^0)$ with coefficients linear in $(D\eta_1(U_0),\cdots, D\eta_N(U_0))$ that are independent of $F.$  Therefore,  
\[
J =\det \frac{\partial\Psi }{\partial U}(0,U_0)=J(H^0_1,\cdots,H^0_N, U_0)
\]
is a polynomial in $(H_1^0,\cdots,H_N^0)$ with coefficients  being polynomials in $(D\eta_1(U_0),\cdots, D\eta_N(U_0))$ and thus rational functions of  $U_0$ that are independent of $F.$ If all these coefficient rational functions  vanish at $U_0,$ then some of them does not vanish at a point in any neighborhood of $U_0,$ and thus $J(H^0_1,\cdots,H^0_N, U_0)$ is not identically zero.   Hence we have  the following:  

\begin{lemma}  There exists $(H_1^0,\dots,H_N^0,U_0)$ such that
\begin{equation}\label{ImFT-1}
\begin{cases}  J(H^0_1,\cdots,H^0_N, U_0)\ne 0,\\
\mbox{(\ref{def-F2}) is satisfied by $\tilde H_j=H_j^0-D^2 F_0(\eta_j^1(U_0)).$} \end{cases}
 \end{equation}
\end{lemma}

We will further adjust $(H_1^0,\dots,H_N^0,U_0)$ later, but for now we assume  that $(H_1^0,\dots,H_N^0,U_0)$   satisfies   (\ref{ImFT-1}) and  let $F$  be defined by  (\ref{def-F}).  Then  $F$ is of the form (\ref{F0})  with a smooth convex $G.$ We also have the following result.

\begin{proposition} \label{def-F4} There exists  a ball $\bar{\mathbb B}_r(0)\subset \R^{2\times n}\times \R^{2\times n}$ and a smooth  function $
U \colon \bar{\mathbb B}_r(0) \to \mathcal N(U_0),$  a neighborhood  of $U_0$ in $\U,$  such that  
 \begin{equation}\label{ImFT-4}
U(0)=U_0,\;\;  
  \Psi (\rho,U(\rho))=0, \; \;  \det \dfrac{\partial\Psi }{\partial U}(\rho,U)\ne 0   
    \end{equation}
    for all $\rho\in\bar{\mathbb B}_r(0)$ and $U\in \mathcal N(U_0).$  
    Moreover,
     \[
    DU(\rho)=-\Big [\frac{\partial \Psi}{\partial U}(\rho,U(\rho))\Big ]^{-1} \frac{\partial \Psi}{\partial \rho}(\rho,U(\rho))  \quad \forall\, \rho\in\bar{\mathbb B}_r(0);
    \]
  in particular,   
    \[
    DU(0)=-\Big [\frac{\partial \Psi}{\partial U}(0,U_0)\Big ]^{-1} \frac{\partial \Psi}{\partial \rho}(0,U_0)  
    \]
  is a rational function of $(H^0_1,\dots,H^0_N)$ with coefficients being rational functions of $U_0$ independent of $F.$ Thus we can adjust $(H_1^0,\dots,H_N^0,U_0)$ so that $DU(0)$ and thus all $DU(\rho)$ will have a full rank of $4n.$
  \end{proposition}
  \begin{proof}
  This follows from the implicit function theorem.
\end{proof}

 \subsection{Toward the verification of Condition $(OC)_N$}   

Let   $\gamma_i(U)$ and $\eta_i(U)$   be  defined in (\ref{aBsg}) and (\ref{eta}) above,  and  let $\kappa_i(U)$ be simply the projection of $U\in \U$ to $\kappa_i$ for $i=1,\cdots,N.$  

With the function $U=U(\rho)$ defined in Proposition \ref{def-F4}, we define the following functions of $\rho\in \bar{\mathbb B}_r(0)$:
 \begin{equation}\label{prop-2}
 \begin{cases}
  \chi_i(\rho)=\frac{1}{\kappa_i(U(\rho))}, \\
  \xi_i(\rho)=\rho+\eta_i(U(\rho)), \\
 \pi_1(\rho)=\rho, \\
  \pi_{i+1}(\rho)=\pi_i(\rho)+\gamma_i(U(\rho)) \end{cases}  \forall\, i \;  \mod N.
 \end{equation}
 Clearly,  these functions  verify Property (P2) of Definition \ref{OC-N} for $\sigma=DF.$

 Since $\pi_i(\rho)=\rho+\gamma_1(U(\rho))+\cdots+\gamma_{i-1}(U(\rho))$, it follows that  
 \[
 \det D\pi_i(0)= \det [I +D(\gamma_1+\cdots+\gamma_{i-1})(U_0)DU(0)]
 \]
  is a polynomial in $(H_1^0,\cdots,H_N^0)$ with coefficients being  rational functions of  $U_0$ that are independent of $F.$ Thus $(H^0_1,\cdots,H^0_N, U_0)$ can be adjusted to satisfy 
\begin{equation}\label{pi-i}
 \det D\pi_i(0)\ne 0\quad \forall\,i=1,\cdots,N.
 \end{equation}
 We choose a further small $r>0$ such that
\begin{equation}\label{pi-i-1}
\begin{cases} \mbox{$\pi_i(\B_r(0))$ is open and $ \det D\pi_i|_{\bar\B_r(0)}\ne 0,$} \\[1ex]
\xi_i^1(\bar\B_r(0))\cap \xi_j^1(\bar\B_r(0)) =\pi_i^1(\bar\B_r(0))\cap\pi_j^1(\bar\B_r(0))=\emptyset
\end{cases}
 \end{equation}
for  all  $i\ne j \in \{1,\cdots,N\};$ this will verify Property (P1).
 
To verify Property (P3), we define, for each fixed $i=1,\cdots,N,$
\[
\begin{cases} z_i(\rho) =\zeta_i(U(\rho))=\kappa_i(U(\rho))\gamma_i(U(\rho)),\\
M_i(\rho)= (D\pi_i(\rho))^{-1}Dz_i(\rho).
\end{cases}
\]
\begin{proposition}\label{pro-n} The matrix $M_i(\rho)$  has an eigenvalue $0$ of multiplicity at least $(n+1)$ and an eigenvalue $-1$ of multiplicity at least $2n,$ and thus   has at most $(n-1)$ eigenvalues not equal to $0$ or $-1.$ We write the characteristic polynomial  of $M_i(\rho)$  as 
\[
 \det [xI-M_i(\rho)] =x^{n+1}(x+1)^{2n} Q(\rho)(x),
 \]
 where $Q(\rho)(x)= x^{n-1} + c_{1}(\rho)x^{n-2}+\cdots+c_{n-1}(\rho),$ with functions $c_j(\rho)$ being polynomials of $M_i(\rho).$   
\end{proposition}
\begin{proof}   Since $ \Phi(\pi_i(\rho)+z_i(\rho))=\Phi(\xi_i(\rho))=0$ and $\rank D\Phi(\xi_i(\rho)) =2n$   for all $\rho\in\B_r(0),$  we have $D\Phi(\xi_i(\rho))D\xi_i(\rho)= 0,$ and thus
  \[
\dim \ker (I+M_i(\rho))=  \dim \ker  D\xi_i(\rho) \ge 2n.
\]
Thus $M_i(\rho)$ has an eigenvalue $-1$ of multiplicity at least $2n.$ By Lemma \ref{rk5},
\[
\rank M_i(\rho)=\rank Dz_i(\rho)\le \rank D\zeta_i(U(\rho))\le 3n-1.
\]
Thus $\dim \ker M_i(\rho)\ge n+1$ and hence $M_i(\rho)$ has an eigenvalue $0$ of multiplicity at least $(n+1).$ Therefore, the number (counted by multiplicity) of eigenvalues of $M_i(\rho)$ other than $0$ and $-1$ is at most $4n-(n+1)-2n=n-1.$ Long divisions show that the coefficients of $Q(\rho)$ are affine functions of the coefficients of polynomial $ \det [xI-M_i(\rho)]$ and hence are polynomials of  $M_i(\rho)$.
\end{proof}

For $\rho\in\B_r(0),$ let $D(\rho)$ be the discriminant of polynomial 
\[
Q(\rho)(x)=x^{n-1} + c_1(\rho)x^{n-2}+\cdots+c_{n-1}(\rho). 
\]
 Then  $D(\rho)$ is a  polynomial  of $(c_1(\rho),\cdots,c_{n-1}(\rho))$ and thus a polynomial of  $M_i(\rho).$ Furthermore,
\[
 Q(\rho)(-1)= (-1)^{n-1} + c_1(\rho)(-1)^{n-2}+\cdots+c_{n-1}(\rho).
\]
is  a polynomial of  $M_i(\rho).$
Thus both $D(0)$and $Q(0)(-1)$ are polynomials of  $M_i(0)$ and hence are rational functions of  $(H^0_1,\cdots,H^0_N, U_0)$. Some calculations can show that these rational functions are not identically zero; thus we adjust  $(H^0_1,\cdots,H^0_N, U_0)$ to  further satisfy that
\begin{equation}\label{need-1}
D(0)\ne 0,\quad Q(0)(-1) \ne 0.
\end{equation} 

Let
\[
E_i(\rho) =\{x<-1: Q(\rho)(x)=0\} 
\]
be the set of eigenvalues of $M_i(\rho)$ less  than $-1.$ 

\begin{lemma} \label{lem-need-1} If $E_i(0) = \emptyset$, then there exists a further small $r>0$  such that $-\frac{1}{\lambda}\notin E_i(\rho)$ for all $\lambda\in (0,1)$ and $\rho\in  \bar\B_r(0).$
\end{lemma}
\begin{proof} Suppose not; then there are $\lambda_j\in (0,1)$ and $\rho_j\to 0$ such that $-\frac{1}{\lambda_j}\in E_i(\rho_j).$ Thus $Q(\rho_j)(-\frac{1}{\lambda_j})=0$ and $ -\frac{1}{\lambda_j}\u_j=M_i(\rho_j)\u_j$ for some unit vector $\u_j.$ Let via a subsequence $\lambda_j\to \lambda\in [0,1]$ and  $\u_j\to \u,$  which yields $-\u=\lambda M_i(0)\u;$ thus $ \lambda \ne 0$ and $Q(0)(-\frac{1}{\lambda})=0.$ If $\lambda\ne 1$ then $-\frac{1}{\lambda}\in E_i(0),$ a contradiction. If $\lambda=1$ then $Q(0)(-1)=0,$   contradicting   (\ref{need-1}).
\end{proof}

 \begin{lemma} 
Assume 
$
E_i(0)=\{x^0_1,x^0_2, \cdots, x^0_k\} \ne\emptyset,
$
where $k\in [1, n-1]$ and $x^0_1<x^0_2<\cdots<x^0_k<-1.$  Then
\[
E_i(\rho)=\{x_1(\rho), \cdots, x_k(\rho)\},
\]
where $x_1(\rho)<\cdots<x_k(\rho)<-1$ are continuous functions on  $\bar\B_r(0)$ for a further small $r>0$  such that $x_j(0)=x^0_j$ for $j=1,\cdots,k.$ 
\end{lemma}

\begin{proof} It is well-known \cite{La} that around each simple zero $x_j^0$ of $Q(0)(x)$  there exists a simple zero $x_j(\rho)$ of $Q(\rho)(x)$, where $x_j(\rho)$ is $C^\infty$ on the coefficients of $Q(\rho)(x)$ and thus is $C^\infty$ on $\rho\in \bar\B_r(0)$ for a further small $r>0$ such that  $x_j(0)=x^0_j$ for $j=1,\cdots,k.$ This proves
\[
 \{x_1(\rho), \cdots, x_k(\rho)\} \subseteq E_i(\rho).
\]
We show that there exists a further small $r>0$ such that 
\[
E_i(\rho)= \{x_1(\rho), \cdots, x_k(\rho)\} \quad \forall\, \rho\in\bar\B_r(0).
\]
Suppose not; then there exist $\rho_j\to 0$ and $y_j\in E_i(\rho_j)$ such that $y_j\notin \{x_1(\rho), \cdots, x_k(\rho)\}.$ Since the sets $\{E_i(\rho_j)\}_{j=1}^k$ are uniformly bounded, we assume $y_j\to \bar y\le -1$ and $Q(0)(\bar y)=0.$ Thus $\bar y\ne -1$ and hence $\bar y\in E_i(0),$ say $\bar y=x_1^0.$ We write
\[
Q(\rho_j)(x)=(x-x_1(\rho_j))\cdots (x-x_k(\rho_j)) (x-y_j)(x^{n-1-k-1}+\cdots ).
\]
Thus $Q(0)(x)=(x-x^0_1 )^2\cdots (x-x^0_k)  (x^{n-1-k-1}+\cdots ),$ which implies $x_1^0$ is a zero of $Q(0)(x)$ of multiplicity at least 2, a contradiction.
\end{proof}

\begin{lemma} \label{lem-need-2}   Assume 
$
E_i(0)=\{x^0_1,x^0_2, \cdots, x^0_k\}\ne\emptyset.
$ We can further adjust $(H^0_1,\cdots,H^0_N, U_0)$ so that
\begin{equation}\label{need-2}
\adj (x_j^0I -M_i(0)) D\pi_i(0) z_i(0)\ne 0 \quad  \forall j=1,\cdots,k.
\end{equation}
\end{lemma}

\begin{proof} Since $M_i(0)$ is a rational (thus analytic) function of $(H^0_1,\cdots,H^0_N, U_0)$, by the analytic implicit function theorem  (see \cite{Ka,To}) it follows that each simple eigenvalue $x_j^0$ of $M_i(0)$ is an analytic function of $(H^0_1,\cdots,H^0_N, U_0)$ and so is  the function
\[
[\adj (x_j^0I -M_i(0))] [\adj D\pi_i(0)] z_i(0).
\]
 The zero set of a nonzero analytic function is a null set. Thus we can further adjust $(H^0_1,\cdots,H^0_N, U_0)$ so that
\[
\adj (x_j^0I -M_i(0)) D\pi_i(0) z_i(0)\ne 0 \quad  \forall j=1,\cdots,k.
\]
\end{proof}
\begin{remark} If $n=2,$ when $E_i(0)\ne \emptyset$, then $E_i(\rho)=\{4+\tr M_i(\rho)\}$ is a singleton. In this case the simple eigenvalue $x^0$ of $M_i(0)$ equals $4+\tr M_i(0)$, which is a rational function of  $(H^0_1,\cdots,H^0_N, U_0).$  \end{remark}

\begin{lemma}\label{imp}  There exists a further small $r>0$  such  that, given any $i\in\{1, \cdots,N\}$, $  \lambda\in (0,1)$ and $\rho\in \bar\B_r(0)$, one has 
\[
\adj [D\pi_i(\rho)+ \lambda  D z_i(\rho)]z_i(\rho) \ne 0
\]
 whenever $\det [D\pi_i(\rho)+ \lambda  D z_i(\rho)]=0.$  
\end{lemma}
\begin{proof} 
Note that $\det [D\pi_i(\rho)+ \lambda  D z_i(\rho)]=0$ for $\lambda\in (0,1)$ if and only if $-\frac{1}{\lambda}\in E_i(\rho).$ If $E_i(0)=\emptyset$ for some $i,$ then by Lemma \ref{lem-need-1},  no such condition is possible for this $i$ and all $\rho\in \bar\B_r(0).$  If $E_i(0)\ne \emptyset$ for some $i$, then by (\ref{need-2}), we take a further small $r>0$  so that
\begin{equation}
[\adj (x_j(\rho)I -M_i(\rho)) ][\adj D\pi_i(\rho)] z_i(\rho)\ne 0  \quad \forall\,\rho\in \bar\B_r(0), \;  \; 1\le j\le k 
\end{equation}
for each $i$ with $E_i(0)\ne\emptyset.$ The result follows because  if $-\frac{1}{\lambda}=x_j(\rho)$ then
\[
\adj [D\pi_i(\rho)+ \lambda  D z_i(\rho)]z_i(\rho)\]
\[
= (-\lambda)^{4n-1} [\adj (x_j(\rho)I -M_i(\rho))] [\adj D\pi_i(\rho) ] z_i(\rho)\ne 0.
\]
\end{proof}

\begin{proposition} The set $\Sigma(\lambda)$ is open for all $0<\lambda\le 1.$
\end{proposition}
\begin{proof} 
  Let $\bar \xi\in \Sigma(\lambda).$ Then for some $i\in\{1,\cdots,N\}$, $0\le \bar\lambda<\lambda$ and  $\bar\rho\in \B_r(0)$, we have 
   \[
 \bar \xi=  \bar\lambda \xi_i(\bar\rho) +(1-\bar\lambda)\pi_i(\bar \rho)=\pi_i(\bar \rho)+\bar\lambda z_i(\bar\rho).
 \]
We show that there exists $\epsilon>0$ such that $\B_\epsilon(\bar \xi)\subset \Sigma(\lambda);$ this proves the openness of $\Sigma(\lambda).$

{\it Case 1:} Assume $\det [D\pi_i(\bar\rho)+ \bar\lambda  D z_i(\bar\rho)] \ne 0.$  Then by the inverse function theorem,  there exist $\epsilon>0$ and $s>0$ with $B_s(\bar\rho)\subset \B_r(0)$ such that
\[
\B_\epsilon(\bar \xi)\subset (\pi_i +\bar\lambda z_i)(\B_s(\bar\rho))\subset (\pi_i +\bar\lambda z_i)(\B_r(0))=S_i(\bar\lambda)\subset \Sigma(\lambda).
\] 

{\it Case 2:}  Assume $\det [D\pi_i(\bar\rho)+ \bar\lambda  D z_i(\bar\rho)] = 0.$ In this case, $\bar\lambda\ne 0;$ thus $0<\bar\lambda<\lambda.$ Consider the function 
 \[
 H(\rho,\xi)= \pi_i(\rho) + [\bar\lambda + (\rho-\bar \rho)\cdot \bar b] z_i(\rho) -\xi 
 \]
 for  all $\rho\in  \bar\B_r(0)$ and  $\xi\in\R^{2\times n}\times \R^{2\times n},$ where $ \bar b= \adj [D\pi_i(\bar\rho)+ \bar\lambda  D z_i(\bar\rho)]z_i(\bar\rho)  \ne 0 $  (by Lemma \ref{imp}). 
 Note that  $H(\bar \rho,\bar \xi)= 0,$ and 
 \[
 \frac{\partial H}{\partial  \rho} (\bar \rho,\bar \xi) =D\pi_i(\bar\rho) +\bar\lambda Dz_i(\bar\rho) +z_i(\bar \rho) \otimes \bar b.
 \]
Thus, 
\[
 \det \frac{\partial H}{\partial  \rho} (\bar \rho,\bar \xi) =\det [D\pi_i(\bar\rho) +\bar\lambda Dz_i(\bar\rho) +z_i(\bar \rho) \otimes \bar b]\]
 \[
 =\adj [D\pi_i(\bar\rho)+ \bar\lambda  D z_i(\bar\rho)]z_i(\bar\rho) \cdot \bar b = |\bar b|^2\ne 0.
\]
Hence  by the implicit function theorem, there exist balls $\B_{s}(\bar\rho)\subset \B_r(0)$ and $\B_\epsilon(\bar\xi)$ such that for all $\xi\in \B_\epsilon(\bar\xi)$ there exists a $\rho\in \B_{s}(\bar\rho)$ such that $H(\rho,\xi)=0$ and, with $s>0$ being further small,
\[
 0<\bar\lambda + (\rho-\bar \rho)\cdot \bar b<\lambda;
 \]
therefore, $\B_\epsilon(\bar\xi)\subset S_i(\bar\lambda + (\rho-\bar \rho)\cdot \bar b)\subset \Sigma(\lambda)$ also holds  in this case. 

 \end{proof}

  \begin{proposition} There exists a number $\delta_1\in (0,1)$ such that the family $\{S_i(\lambda)\}_{i=1}^N$ is disjoint and open for all $\delta_1\le  \lambda <1.$
  \end{proposition}
\begin{proof} Let
\[
x_0=\max_{1\le i\le N, \, E_i(0)\ne\emptyset, \, \rho\in \bar\B_r(0)} E_i(\rho).
\]
Then $x_0<-1.$ Let $\delta_1\in (-\frac{1}{x_0},1).$ Then $\delta_1\in (0,1).$ For all $\lambda\in [\delta_1,1),$  it follows that $-\frac{1}{\lambda}$ is not an eigenvalue of $M_i(\rho)$ for any $i$ and $\rho.$   Thus
\[
\det [D\pi_i(\rho)+ \lambda  D z_i(\rho)] =\det D\pi_i( \rho) \det [I+\lambda M_i( \rho)] \ne 0
\]
for all $ 1\le i\le N$, $\rho\in\B_r(0)$ and  $\lambda\in [\delta_1,1).$ Hence, by the inverse function theorem, $S_i(\lambda)$ is open all $\lambda\in [\delta_1,1).$  Finally, by (\ref{pi-i-1}), we can choose $\delta_1$ even closer to 1 such that $\{S_i(\lambda)\}_{i=1}^N$ is disjoint for all $ \lambda \in [\delta_1, 1).$
\end{proof}

\end{document}